\documentclass[12pt]{article}
\usepackage{latexsym}
\usepackage{amssymb}
\usepackage{amsthm}
\usepackage{graphicx} 
\usepackage{enumerate}

\newtheorem{theorem}{Theorem}[section]
\newtheorem{definition}[theorem]{Definition}
\newtheorem{lemma}[theorem]{Lemma}

\def\setR{\ensuremath{\mathbb{R}}}
\def\setC{\ensuremath{\mathbb{C}}}

\newcommand{\delbar}{{\ensuremath{\bar\partial }}}

\def\del{{\partial}}
\def\id{\mathrm{Id}}
\def\to{\longrightarrow}

\def\M{\mathcal{M}}

\def\H{\ensuremath{\mathcal{H}}} 
\def\Lie{\mathcal{L}}  
\newcommand{\norm}[1]{{\ensuremath{|\!|#1|\!|}}}

\title{Embedded $\H$--Holomorphic Maps and\\ Open Book Decompositions}
\author{Jens von Bergmann}

\begin{document}
\maketitle {\abstract{We investigate nicely embedded $\H$--holomorphic
    maps into stable Hamiltonian three--manifolds. In particular we
    prove that such maps locally foliate and satisfy a
    no--first--intersection property. Using the compactness results of
    \cite{h-hol_compactness} we show that connected components of the
    space of such maps can be compactified if they contain a global
    surface of section. As an application we prove that any contact
    structure on a 3--manifold admits and $\H$--holomorphic open book
    decomposition.  This work is motivated by the program laid out by
    Abbas, Cieliebak and Hofer in \cite{planar_weinstein} to give a
    proof to the Weinstein conjecture using holomorphic curves.  The
    results in this paper, with the exception of the compactness
    statement, have been independently obtained by C. Abbas
    \cite{abbas_solutions}.}  }

\section{Introduction}
\label{sec:introduction}
Let $(Z,\alpha,\omega,J)$ be a stable Hamiltonian structure on an
oriented closed $3$--manifold $Z$, i.e. $\alpha$ is a 1--form and
$\omega$ is a closed 2--form such that $\alpha\wedge\omega>0$ and
$\ker\omega\subset\ker d\alpha$. This induces a splitting $TZ=L\oplus
F$ where $L=\ker(\omega)$ is called the characteristic foliation and
$F=\ker(\alpha)$ is the almost contact plane field. $L$ has a
distinguished section $R$ defined by $\alpha(R)=1$ called the
characteristic vector field, and $J\in End(F,\omega)$ is a choice of
$\omega$--compatible almost complex structure on $F$. We extend $J$ to
all of $TZ$ by precomposing with the projection
$\pi_F=\id-R\otimes\alpha$ along $R$ onto $F$. For more details we
refer to \cite{h-hol_compactness}.

\begin{definition}[$\H$--Holomorphic Maps]\label{def:H_hol} 
  Let $(\dot\Sigma,j)$ be a punctured Riemann surface. A map $v:\dot
  \Sigma\to Z$ is called {\em $\H$--holomorphic} if
  \begin{eqnarray} 
    \delbar^F_J v&=&0,\qquad\delbar_J^F=\frac{1}{2}\left(\pi_F\,dv+J\,\pi_F\,dv\,j\right)
    \label{eq:H_F}\\
    d(v^\ast\alpha\circ j)&=&0,\label{eq:H_L}\\
    \int_{\del B_p(\varepsilon)}v^\ast\alpha\circ j&=&0\qquad 
    \forall\;p\in\Sigma\setminus\dot\Sigma\quad\mathrm{and}\ \varepsilon\
    \mathrm{small\ enough}.\label{eq:H_per}
  \end{eqnarray}
\end{definition}

$\H$--holomorphic maps were introduced by Abbas, Cieliebak and Hofer
\cite{planar_weinstein} with the purpose to produce $\H$--holomorphic
open book decompositions and use these to prove the Weinstein
Conjecture in three dimensions, which has subsequently been proved by
Taubes \cite{taubes_weinstein}. The purpose of this work is to prove
some of the steps in the direction indicated in
\cite{planar_weinstein}. In Theorem \ref{thm:open_book} we show that
indeed every contact structure is supported by an $\H$--holomorphic
open book decomposition. We moreover prove that nicely embedded
$\H$--holomorphic maps locally foliate stable Hamiltonian
three--manifolds (Theorem \ref{thm:local_foliation}), and that
connected components of the space of nicely embedded $\H$--holomorphic
maps can be compactified if one of the maps is a global surface of
section (Theorem \ref{thm:moduli_space}).

Compactness of the space of $\H$--holomorphic maps from a genus $g$
surface $\Sigma$ is a delicate issue. $\H$--holomorphic maps can be
viewed as a parameter version of $J$--holomorphic maps with
$2g$--dimensional parameter space given by $H^1(\Sigma;\setR)$. The
parameter space is not compact, causing connected components of the
space of $\H$--holomorphic maps to be in general not compact. However,
under certain topological and geometric assumptions, the space of
$\H$--holomorphic maps does have a natural compactification. These
results were established in \cite{h-hol_compactness}, where a
general criterion for compactness of $\H$--holomorphic maps was given
that we will make use of and restate here for the convenience of the reader.
\begin{theorem}[\cite{h-hol_compactness}]\label{thm:compactness}
  Let $(Z,\alpha,\omega,J)$ be a stable Hamiltonian structure so that
  all closed characteristics are Morse or Morse--Bott. The space of
  smooth $\H$-holomorphic maps into $Z$ with uniformly bounded
  $\omega$ and $\alpha$--energies with uniformly bounded periods 
  has compact closure in the space of neck--nodal $\H$--holomorphic
  maps.
\end{theorem}
Bounded periods means that for each free homotopy class of simple
closed loops there exists a constant $C>0$ so that for each
$\H$--holomorphic map $v:(\dot\Sigma,j)\to Z$ in the family
\begin{eqnarray*}
    P_{[\gamma]}(v)=\sup_{\phi\in\Phi(j,[\gamma])}\sup_{s\in(0,1)}
      \left|\int_{\sigma_s(\phi)}v^\ast\alpha\right|<C,
\end{eqnarray*}
where $\Phi(j,[\gamma])$ is the set of 1--cylinder
Strebel--differentials associated to $j$ and the free homotopy class
of a simple closed loop $\gamma$ and $\sigma_s$ denote the closed
leaves of $\phi$.  The period integrals are essentially controlled by
the harmonic part of $v^\ast\alpha$.  For more details see
\cite{h-hol_compactness}.

An $\H$--holomorphic map $v:\dot\Sigma\to Z$ is locally
$J$--holomorphic, so $v$ has the same local properties as a
$J$--holomorphic map into $Z$. In particular, in local conformal
coordinates $C=[0,\infty)\times S^1$ near each puncture $p$ the map
$v$ is asymptotic to a closed characteristic $x:S^1\to Z$ in $Z$ and
has transverse approach governed by an eigenvector $e$ of the
asymptotic operator (see Equation (\ref{eq:asymptotic_operator})) with
eigenvalue $\lambda<0$. We say that the transverse asymptotic approach
is {\em maximal} if $\lambda$ is the largest negative eigenvalue of
the asymptotic operator, and $-\lambda<2\pi$. For simplicity we will
assume throughout that the transverse approach of the maps in question
is maximal at all punctures.  Throughout we work with weighted Sobolev
spaces with weights $\delta>0$ chosen small enough so that
$\delta<-\lambda$ at all punctures. The importance of the assumption
that $-\lambda<2\pi$ is important for the interplay between the
``tangent'' and ``normal'' operators in the asymptotic analysis and
will become clear in Section \ref{sec:linearized-equation}. For more
details on the asymptotic operator and the choice of Sobolev spaces we
refer to \cite{dragnev}.

In this work we will focus on nicely embedded $\H$--holomorphic maps
into stable Hamiltonian three--manifolds.
\begin{definition}
  An immersed $\H$--holomorphic map $v:\dot\Sigma\to Z$ is called {\em
    nicely immersed} if the transverse approach at all punctures is
  maximal and the Fredholm index of $v$ is 1.

  An embedded nicely immersed map is called {\em nicely embedded} if
  the map extends to an embedding over the radial compactification;
  in particular all asymptotic orbits are covered once.
\end{definition}
This definition is roughly modeled on the definition of
nicely embedded $J$--holomorphic spheres given in
\cite{wendl_transversality}. 

It follows from standard index formulas (see
e.g. \cite{wendl_transversality}, remembering that the index for
$\H$--holomorphic maps is $2g-1$ higher than the index of a
$J$--holomorphic map into the symplectization) it means for an
immersed $\H$--holomorphic map $v$ to be nicely immersed is that the
all asymptotic orbits have odd Conley--Zehnder index. In other words,
if $\lambda$ is the eigenvalue of the eigenvector of the asymptotic
operator governing the transverse approach at any of the punctures,
then the other eigenvector with the same winding has eigenvalue
$\hat\lambda\le\lambda$.

The argument in this paper can be roughly split into three parts. The
first part considers the local Fredholm theory yielding transversality
and a strong implicit function theorem and thus showing that nicely
embedded maps locally foliate. The second part proves a weak version
of global intersection theory recently developed in more generality by
R. Siefring in \cite{siefring_intersection} yielding a
no--first-intersections result. We combines these results with the
compactness result from Theorem \ref{thm:compactness} to prove that
the space of nicely embedded maps can be compactified. The third part
is an application showing that every contact structure is supported by
an $\H$--holomorphic open book decomposition. 

With the exception of the compactness results, these results have been
independently established by C. Abbas \cite{abbas_solutions}.

\section{The Linearized Equation}
\label{sec:linearized-equation}

Suppose $v:\dot\Sigma\to Z$ is a smooth embedded $\H$--holomorphic
map. We want to linearize the equations at $v$. We will work with the
(generalized) Tanaka--Webster connection $\tilde\nabla$ and we quickly
recall the properties of $\tilde \nabla$ that we will make use of.
For details we refer to \cite{tanno}. Let $X$ and $Y$ be sections of
$F$. Then
\begin{eqnarray*}
  \tilde\nabla \alpha=0,\quad\tilde\nabla R=0,\quad\tilde\nabla g=0\\
  (\tilde\nabla_RJ)=0,\quad (\tilde\nabla_XJ)R=-X-\phi X,\quad
\\
  \tilde T(R,R)=0,\quad\tilde T(R,JX)=-J\tilde T(R,X)=-\phi X,\quad
\end{eqnarray*}
where $\phi=\frac12\Lie_RJ$.

Given $\xi\in\Omega^0(\dot\Sigma,v^\ast TZ)$ we write
$\xi=\zeta\, R+\chi$ with $\chi$ a section of $v^\ast F$ and set
\begin{eqnarray*}
  \Phi_v:v^\ast TZ\to \exp_v(\xi)^\ast TZ
\end{eqnarray*}
denote parallel transport along the geodesic $s\mapsto
\exp_{v(z)}(s\xi(z))$ with respect to $\tilde\nabla$.

The space of diffeomorphisms of $\dot\Sigma$ with $N$ (fixed)
punctures acts on the space of maps in the usual way.  Locally at a
map $(v,j)$ consider we fix a $6g-6+2N$--dimensional slice
$T_j\mathcal T$ of infinitesimal variations of complex structure
$h$. We will assume that the slice is chosen so that the variations of
complex structure have support away from the punctures.

Then define
\begin{eqnarray*}
  \mathcal{F}_v^F&:&
  \Omega^0(\dot\Sigma,v^\ast TZ)\times T_j\mathcal{T}\to 
  \Omega^{0,1}(\dot\Sigma,v^\ast TZ)\\
  \mathcal{F}_v^F(\xi,h)&=&\Phi_v(\xi)^{-1}
  \frac12\left(\pi_F\,d\exp_v(\xi)+J\pi_Fd\exp_v(\xi)\circ j_h\right)
\end{eqnarray*}
and
\begin{eqnarray*}
  \mathcal{F}_v^L
  &:&\Omega^0(\dot\Sigma,v^\ast TZ)\times T_j\mathcal{T}\to 
  \Omega^{2}(\dot\Sigma)\\
  \mathcal{F}_v^L(\xi,h)&=&d(\Phi_v(\xi)^\ast\alpha\circ j_h)
\end{eqnarray*}
where $j_h$ is a variation of complex structure satisfying
$\frac{d}{dt}j_{th}|_{t=0}=h$. Such a family of variation can be chosen
canonically, see e.g. Section 3.2 of \cite{wendl_transversality}.

\begin{lemma}\label{lem:linearization}
  For any smooth map $v:\dot\Sigma\to M$ asymptotic to closed
  characteristics at the punctures, define the operators
  \begin{eqnarray*}
    D^F_v(\xi,h)=d\mathcal{F}^F_v(0)(\xi,h),\qquad
    D^L_v(\xi,h)=d\mathcal{F}^L_v(0)(\xi,h).
  \end{eqnarray*}
  Then
  \begin{eqnarray}
    D^F_v(\xi,h)&=&\tilde\nabla^{0,1}\chi
    +\frac12(\tilde\nabla_\chi J)\pi_F\,dv\circ j
    +\frac12J\pi_F\,dv\circ h
    +(J\phi\chi\otimes v^\ast\alpha)^{0,1}
    -\zeta J\phi\pi_F\,dv\label{eq:lin_F}\\
    D^L_v(\xi,h)&=&d\left[d\zeta\circ j +v^\ast(\iota_\chi d\alpha)\circ j
      +v^\ast\alpha\circ h\right].\label{eq:lin_L}
  \end{eqnarray}
\end{lemma}

\begin{proof}
  Consider the path $\setR\to C^{\infty}(\dot\Sigma,M)$ given
  by $s\mapsto v_s=\exp_v(s\xi)$ and a path $j_s$ with
  $\frac{d}{ds}j_s|_{s=0}=h$. Then
  \begin{eqnarray*}
    D_v^F(\xi,h)
    &=&\frac{d}{ds}\mathcal{F}_v^F(s\xi,sh)\big|_{s=0}\\
    &=&\frac{d}{ds}\Phi_v(s\xi)^{-1}
    \frac12\left(\pi_F\,dv_s+J\pi_Fdv_s\circ j_s\right)\big|_{s=0}\\
    &=&\frac12\tilde\nabla_s
    \left(\pi_F\,dv_s+J\pi_Fdv_s\circ j_s\right)\big|_{s=0}\\
    &=&\frac12
    \left(\pi_F\,(\tilde\nabla_sdv_s)+(\tilde\nabla_\xi
      J)\pi_Fdv\circ j+J\pi_F(\tilde\nabla_s dv_s)\circ
      j+J\pi_F\,dv\circ h \right)_{s=0}\\
    &=&\frac12
    \Big(\pi_F\,(\tilde\nabla\xi+\tilde T(\xi,dv))+(\tilde\nabla_\xi
      J)\pi_Fdv\circ j+J\pi_F(\tilde\nabla\xi+\tilde T(\xi,dv))\circ
      j\\
      &&+J\pi_F\,dv\circ h \Big)\\
    &=&\tilde\nabla^{0,1}\chi+(\pi_F\tilde T(\xi,dv))^{0,1}
    +\frac12(\tilde\nabla_\chi J)\pi_F\,dv\circ j
    +\frac12J\pi_F\,dv\circ h\\
    &=&\tilde\nabla^{0,1}\chi
    +\frac12(\tilde\nabla_\chi J)\pi_F\,dv\circ j
    +\frac12J\pi_F\,dv\circ h\\
    &&+(\tilde T(\chi,R)\otimes v^\ast\alpha)^{0,1}
    +\zeta\tilde T(R,\pi_F\,dv)\\
    &=&\tilde\nabla^{0,1}\chi
    +\frac12(\tilde\nabla_\chi J)\pi_F\,dv\circ j
    +\frac12J\pi_F\,dv\circ h
    +(J\phi\chi\otimes v^\ast\alpha)^{0,1}
    -\zeta J\phi\pi_F\,dv.
  \end{eqnarray*}
  Also
  \begin{eqnarray*}
    D_v^L(\xi,h)
    &=&\frac{d}{ds}\mathcal{F}_v^L(s\xi,sh)\big|_{s=0}
    =\frac{d}{ds} d(\exp_v(s\xi)^\ast\alpha\circ j_s)\big|_{s=0}
    =d\left((\Lie_\xi\alpha)\circ j+v^\ast\alpha\circ h\right)\\
    &=&d\left[d\zeta\circ j +v^\ast(\iota_\chi d\alpha)\circ j
    +v^\ast\alpha\circ h\right].
  \end{eqnarray*}
\end{proof}

We choose the usual functional analytic setup for the spaces of maps,
i.e. at an $\H$--holomorphic map $v:\dot\Sigma\to Z$ we
consider the linearized operator to act on
\begin{eqnarray*}
  D^F_v&:&W^{k,p}_\delta\left(\Omega^0(\dot\Sigma,v^\ast TZ)\right)\times
    T_j\mathcal{T} \to 
  W^{k-1,p}_\delta\left(\Omega^{0,1}(\dot\Sigma,v^\ast TZ)\right)\\
 D^L_v&:&W^{k,p}_\delta\left(\Omega^0(\dot\Sigma,v^\ast
   TZ)\right)\times \setR^N
 \times T_j\mathcal{T}\to  {W}^{k-2,p}_\delta\Omega^{2}(\dot\Sigma)
\end{eqnarray*}
where the weight $\delta>0$ is not in the spectrum of the asymptotic
operator at the puncture and $N$ is the number of punctures. The
$\setR^N$--factor in the linearization allows for functions that are
asymptotically constant near each puncture. For a precise definition
of this and the accompanying space of maps we refer the reader to
\cite{dragnev} or \cite{folded_holomorphic}.

Usually we view the space of gauge transformations to be infinitesimal
diffeomorphism extending smoothly over the punctures, or vanishing at
the punctures. In our case it is at times convenient to alternatively
consider more general gauge transformations that don't necessarily
extend over the punctures, but remain bounded in the cylindrical
metric.

Let $v:C_0=[0,\infty)\times S^1\to Z$ be a $J$--holomorphic map, with
maximal transverse approach given by an eigenvector of the asymptotic
operator with eigenvalue $-2\pi<\lambda<0$ and choose $\delta>0$ so that
$-\delta>\lambda$. Let $\nu$ be a vector field on the domain. Then
$\pi_F\,dv(\nu)\in W^{k,p}_\delta(v^\ast F)$ if and only if $\nu\in
W^{k,p}_{\tilde\delta}(TC_0)$, where
\begin{eqnarray}\label{tilde_delta}
\tilde\delta=\delta+\lambda<0
\end{eqnarray}
satisfies $-\tilde\delta<2\pi$.  We will use this notation for the
weights $\delta$ and $\tilde\delta$ related via the maximal eigenvalue
$\lambda$ for the remaining of this and the following section.

The following Lemma gives an equivalent description to solutions of
the linearized equation at a nicely immersed map.
\begin{lemma}\label{lem:lin_op2}
  Let $v:\dot\Sigma\to Z$ be a nicely immersed $\H$--holomorphic
  map. Then $(\xi,h)$, where $\xi\in W^{k,p}_\delta(v^\ast
  TZ)\times\setR^N$ and $h$ is a smooth deformation of complex
  structure, solves the linearized equation if and only if there
  exists an infinitesimal gauge transformation $\nu\in
  W^{k,p}_{\delta+\lambda}(T\dot\Sigma)$ and $f\in
  W^{k,p}_{\delta+\lambda}(\dot\Sigma,\setR)$ so that
  $\xi=f\,R+dv(\nu)$ and
  \begin{eqnarray}
    \tilde Df=d(df\circ j+2f\,v^\ast\alpha \circ \tilde\phi)=0\label{eq:tildeD}
  \end{eqnarray}
  where
  \begin{eqnarray*}
    \tilde\phi=(\pi_F\,dv)^\ast \phi=(\pi_F\,dv)^{-1}\circ
    \frac12\Lie_RJ\circ(\pi_F\,dv).
  \end{eqnarray*}
\end{lemma}

\begin{proof}
  Since $v$ is immersed we can rewrite the linearized Equation (\ref{eq:lin_F}) with
  $\xi=dv(\nu)+f\,R$, where $\nu$ is a section of $T\dot\Sigma$
  and $f:\dot\Sigma\to\setR$. If $(\xi,h)$ solves the linearized
  equations, then
  \begin{eqnarray}\label{eq:nu}
    \nabla^{0,1}\nu-f j\tilde \phi+\frac12 jh=0.
  \end{eqnarray}
  Solving this for $h$
  \begin{eqnarray*}
    h=2(j\nabla^{0,1}\nu+f\tilde\phi)=\Lie_\nu j+2f\tilde\phi
  \end{eqnarray*}
  and plugging this into Equation (\ref{eq:lin_L}) we obtain
  \begin{eqnarray*}
    d(df\circ j+2f\,v^\ast\alpha\circ\tilde\phi)=0,
  \end{eqnarray*}
  after noting that
  \begin{eqnarray*}
   d( v^\ast\alpha(\nu))\circ j+(\iota_\nu v^\ast d\alpha)\circ j+v^\ast\alpha(\Lie_\nu j)
    =\Lie_\nu (v^\ast\alpha\circ j)
  \end{eqnarray*}
  is closed and that $\zeta$ in Equation (\ref{eq:lin_L}) corresponds to
  $\zeta=f+v^\ast\alpha(\nu)$.

  If $\xi\in W^{k,p}_{\delta}(v^\ast TZ)\times\setR^N$, then $\nu\in
  W^{k,p}_{\delta+\lambda}(T\dot\Sigma)$ and $f\in
  W^{k,p}_{\delta+\lambda}(\dot\Sigma,\setR)$.

  Conversely, if $\nu\in W^{k,p}_{\delta+\lambda}(T\dot\Sigma)$ and
  $f\in W^{k,p}_{\delta+\lambda}(\dot\Sigma,\setR)$, then
  $\xi=dv(\nu)+f\,R\in W^{k,p}_{\delta}(v^\ast TF)\times
  W^{k,p}_{\delta+\lambda}(v^\ast L)$, and $(\xi,h)$ satisfies the
  linearized equations. Noting that $-\tilde\delta<2\pi$ by assumption
  we employ the standard asymptotic analysis for $J$--holomorphic
  curves near the puncture to split the asymptotic operator (see
  Section 3.4 of \cite{wendl_transversality}), which is in upper
  triangular form, and see that $\xi$ is automatically of class
  $W^{k,p}_{\delta}(v^\ast TZ)\times\setR^N$.
\end{proof}

\section{Local Theory for Nicely Immersed Maps}
\label{sec:local-fredh-theory}
In this section we discuss the properties of nicely immersed
$\H$--holomorphic maps into stable Hamiltonian 3--manifolds.  As we
will see, these curves have very nice properties and are well--suited
to yield finite energy foliations.  The key result in this section is
the observation that a strong version of the implicit function theorem
\ref{thm:local_foliation} holds for nicely immersed $\H$--holomorphic
maps.

We need an expression for the asymptotic operator $A_\infty$
associated to an $\H$--holomorphic map $v$ asymptotic to a closed
characteristic $x:S^1\to Z$ of period $T$ in the trivialization given
by the eigenvector $e$ with eigenvector $\lambda<0$ associated to $v$.
Acting on sections $\eta$ of $x^\ast F$ we have
\begin{eqnarray}\label{eq:asymptotic_operator}
  A_\infty\eta=-J(\nabla_t \eta-T\nabla_\eta R)
  =-J\nabla_t \eta+T(\eta+\phi\eta),
\end{eqnarray}
where $\nabla$ denotes the Levi--Civita connection and we used that
$\nabla_X R=-JX-J\phi X$ and $\eta\in F$, where
$\phi=\frac12\Lie_{R}J$. Then with $\eta=z\,e$, $z=x+Jy$ we get
\begin{eqnarray*}
  A_\infty (z\,e)&=&-Jz\nabla_t e-J\nabla_t(z)e+Tz\,e+T\bar z\phi e\\
  &=&z(-J\nabla_t e+T(e+\phi e))-J\dot z\,e-T(z-\bar z)\phi e\\
  &=&(-J\dot z+\lambda z-T(z-\bar z)\phi) e,\\
\end{eqnarray*}
where we used that $\nabla_R J=0$ and that $\phi$ is $J$
anti--linear. Using also that $\phi$ is symmetric (see e.g. Section
6.2 in \cite{blair}) we write $\phi$ in matrix form with respect to
the trivialization given by $e$ and $Je$
\begin{eqnarray}\label{eq:hat_phi}
  \hat\phi=\left[
    \begin{array}{cc}
      \mu_1&\mu_2\\
      \mu_2&-\mu_1
    \end{array}\right].
\end{eqnarray}
Set $\tilde\mu=\mu_1+i\mu_2$. Then we can express the operator
$A_\infty$ in the trivialization given by $e$ as
\begin{eqnarray}\label{eq:a_infty}
  \tilde A_\infty z=-i\dot z+\lambda z-\,T(z-\bar z)\tilde \mu.
\end{eqnarray}

\begin{lemma}\label{lem:lin_op3}
  Let $v:\dot\Sigma\to Z$ be a nicely immersed $\H$--holomorphic map
  and let $f\in W^{k,p}_{\tilde\delta}(\dot\Sigma,\setR)$ satisfy
  $\tilde Df=0$, where $\tilde D$ is the operator from Equation
  (\ref{eq:tildeD}).

  With the same notation for $e$ and $\lambda$ as above, at each
  puncture $f$ is given by the imaginary part (expressed in the
  trivialization given by $e$ and $J\,e$) of an eigenvector $\tilde e$
  of the asymptotic operator $A_\infty$ with eigenvalue
  $\tilde\lambda\le\lambda$, or $f\equiv 0$.
\end{lemma}
\begin{proof}
  We need to understand the asymptotics of the operator $\tilde D$ in
  terms of the asymptotic operator $A_\infty$. Let
  $\kappa=v^\ast\alpha\circ\tilde\phi$. In the neck regions adjacent
  to the punctures write
  \begin{eqnarray*}
    df\circ j+2f\kappa=da.
  \end{eqnarray*}
  evaluating this on $\del_s$ and $\del_t$ we get
  \begin{eqnarray*}
    f_t+2f\kappa(\del_s)=a_s,\qquad -f_s+2f\kappa(\del_t)=a_t.
  \end{eqnarray*}
  with $z=a+if$ and $\tilde\kappa=\kappa(\del_t)-i\kappa(\del_s)$ this gives
  \begin{eqnarray}\label{eq:complex_f}
    z_s=-iz_t+2f(\kappa(\del_s)+i\kappa(\del_t))
    =-iz_t-i(z-\bar z)(\kappa(\del_s)+i\kappa(\del_t))
    =-iz_t+(z-\bar z)\tilde\kappa.
  \end{eqnarray}
  
  We wish to understand solutions of this equation as $z\to\infty$.
  Following the usual asymptotic analysis \cite{hofer_asymptotics},
  denote the $L^2(x^\ast F)$ inner product, and the induced inner
  product on the trivialization, by $\langle.,.\rangle$, i.e. for
  $u,v$ sections of $x^\ast F$
  \begin{eqnarray*}
    \langle u,v\rangle=\int_{S^1}g(u,v)d\theta
  \end{eqnarray*}
  and in the trivialization for $z,w$ complex values functions on $S^1$
  \begin{eqnarray*}
    \langle z,w\rangle= \int_{S^1}g(z\,e,w\,e)d\theta.
  \end{eqnarray*}
  Denote the associated norm by $\norm{.}$ and set
  $\zeta=\frac{z}{\norm{z}}$.  Then
  \begin{eqnarray*}
    \zeta_t=\frac{z_t}{\norm{z}},\qquad\zeta_s=\frac{z_s}{\norm{z}}-\alpha
    \zeta,\quad\mathrm{where}\ \alpha(s)=\frac{\langle z_s,z\rangle}{\norm{z}^2}.
  \end{eqnarray*}
  Then $\zeta$ satisfies
  \begin{eqnarray*}
    \zeta_s=B\zeta+\alpha\zeta,\qquad B\zeta=-i\zeta_t+(\zeta-\bar\zeta)\tilde\kappa
  \end{eqnarray*}
  Now consider the behavior of the terms in $B$ as $s\to\infty$.  We have
  $v^\ast\alpha\to T\,dt$. To understand the behavior of $\tilde\phi$,
  recall that (see Section 4 in \cite{hofer_fredholm})
  \begin{eqnarray*}
    \frac{\pi_F\,dv(\del_s)}{\norm{\pi_F\,dv(\del_s)}}\to \frac{e}{\norm{e}},\qquad
    \mathrm{as}\ s\to\infty
  \end{eqnarray*}
  where $e$ is the eigenvector of $A_\infty$ with eigenvalue $\lambda$
  governing the transverse approach of $v$ at the puncture.  So in the
  basis $\del_s$ and $\del_t$ we have $\tilde \phi\to\hat\phi$, the
  matrix from Equation (\ref{eq:hat_phi}). Thus $\kappa(\del_s)\to T\mu_2$
  and $\kappa(\del_t)\to -T\mu_1$, so $\tilde\kappa\to
  -T(\mu_1+i\mu_2)=-T\tilde\mu$.  Then the operator $B$ approaches
  \begin{eqnarray*}
    B_\infty \zeta=-i\zeta_t-T(\zeta-\bar\zeta)\tilde\mu,
  \end{eqnarray*}
  and the same asymptotic analysis as in \cite{hofer_asymptotics}
  shows that solutions to Equation (\ref{eq:complex_f}) converge to
  eigenvectors of $B_\infty$ exponentially fast, with rate governed by
  the corresponding eigenvalue.  Comparing with Equation
  (\ref{eq:a_infty}) we see that $B_\infty=A_\infty-\lambda\id$, so
  the spectrum of $B_\infty$ is the spectrum of $A_\infty$ shifted by
  $-\lambda$ with identical eigenvectors.

  Since $A_\infty$ has no eigenvalue in the interval
  $(\lambda,-\delta]$, and has one eigenvalue $\lambda$ we have that
  $B_\infty$ has no eigenvalue in the interval $(0,-\tilde\delta]$,
  and has an eigenvalue 0. 

  Let $z$ be a solution to Equation (\ref{eq:complex_f}) of class
  $W^{k,p}_{\tilde\delta}$. Write $z=c+\tilde z$, where $c$ is a real
  constant and $\tilde z$ has vanishing average real part at the
  puncture. Then $\tilde z$ either vanishes identically or is an
  eigenvector of $\hat A_\infty$ with nonpositive eigenvalue and the
  imaginary part of $\tilde z$ does not vanish identically.
\end{proof}

After understanding the asymptotic behavior of elements in the kernel
of $\tilde D$ we prove a transversality and non--vanishing result for
nicely immersed curves. Recall our notation that
$\tilde\delta=\delta+\lambda<0$, where $\delta$ is the weight at the
punctures and $\lambda$ is the eigenvalue governing the asymptotic
approach.
\begin{lemma}\label{lem:positive_solution}
  Let $v$ be a nicely immersed $\H$--holomorphic map. Then the
  operator $\tilde D$ from Equation (\ref{eq:tildeD}) has index 1, is
  surjective and any non--trivial solution $f$ to $\tilde Df=0$ has no zeros.
\end{lemma}

\begin{proof}
  First we consider a neighborhood of each puncture.  Let $\lambda$ be
  the eigenvalue of the eigenvector $e$ governing the transverse
  approach at the puncture. The eigenvector $\hat e$ of the asymptotic
  operator with maximal eigenvalue $\hat\lambda\le\lambda$ has the
  same winding as $e$. Recall that $\hat e$ and $e$ are pointwise
  linearly independent (see Lemma 3.5 of \cite{hofer_embedding}).
  
  We aim to show that solutions to Equation (\ref{eq:complex_f}) are
  necessarily asymptotic to the imaginary part $\hat e$ (in the
  trivialization given by $e$ and $Je$), after subtracting off a
  constant real part. We will do this by relating the kernel of
  $\tilde D$ to the kernel of the associated system of first order
  equations (in complex notation)
  \begin{eqnarray*}
    L:W^{k,p}_{\tilde\delta}(\dot\Sigma,\setC)\times \tilde\H^{0,1}_\setC\to
    W^{k-1,p}_{\tilde\delta}(T^{0,1}\dot\Sigma),\qquad
    L(z,\eta)=\delbar z-(z-\bar z)\kappa^{0,1}+\eta.
  \end{eqnarray*}

  First recall that 
  \begin{eqnarray*}
    \tilde D:W^{k,p}_{\tilde\delta}(\dot\Sigma,\setR)\to
    W^{k-2,p}_{\tilde\delta}(\dot\Sigma),\qquad
    \tilde D f=\ast d(df\circ j+2f\,v^\ast\alpha\circ\tilde\phi).
  \end{eqnarray*}
  is Fredholm if and only if 
  \begin{eqnarray*}
    D_\infty(\lambda):W^{k,p}_{\tilde\delta}(S^1,\setR)\to W^{k-2,p}_{\tilde\delta}(S^1,\setR),\qquad
    D_\infty (\lambda)=-\lambda^2f-f_{tt}+2T(\lambda \mu_2f+(f\mu_1)_t)
  \end{eqnarray*}
  is an isomorphism (c.f. \cite{lockhart_mcowen}), and that the
  Fredholm index is constant on connected components of $\setR$ on
  which $D_\infty$ is an isomorphism. 

  We claim that the set of $\mu$ for which $D_\infty$ is an
  isomorphism coincides with the spectrum of $B_\infty$, with the
  exception of the eigenvalue 0 of $B_\infty$. If $\mu$ is an
  eigenvalue of $B_\infty$ with eigenvector $\zeta$, then a
  straightforward calculation shows that $D_\infty(\mu)f=0$, where
  $f=\Im(\zeta)$. If $\mu\ne 0$, then $f$ is non--trivial (by
  Lemma 3.5 of \cite{hofer_embedding}), and if $\mu=0$, then
  $D_\infty$ is an isomorphism since by assumption the eigenvalue of
  $B_\infty$ has multiplicity 1 and the eigenvector is purely real.

 Conversely, if $\mu\ne 0$ and $D_\infty(\mu)f=0$ then
  $f\not\equiv 0$ and
  \begin{eqnarray*}
    \zeta=a+if,\qquad a=\frac1\mu(f_t+2T\mu_2f)
  \end{eqnarray*}
  is an eigenvector of $B_\infty$ with eigenvalue $\mu$.

  The adjoint $\tilde D^\ast$ of $\tilde D$ is given by
  \begin{eqnarray*}
    \tilde D^\ast:W^{-k+2,p}_{-\tilde\delta}(T^\ast\dot\Sigma)\to 
    W^{-k,p}_{-\tilde\delta}(\Lambda^2\dot\Sigma),\qquad
    \tilde D^\ast g
    =\ast d(dg\circ j)-2dg\wedge v^\ast\alpha\circ\tilde\phi,
  \end{eqnarray*}
  where $\ast$ is taken with respect to the cylindrical metric. By
  elliptic regularity the kernel of $\tilde D^\ast$ is the same as the
  kernel of
  \begin{eqnarray*}
      \tilde D^\ast:W^{l,p}_{-\tilde\delta}(T^\ast\dot\Sigma)\to 
    W^{l-2,p}_{-\tilde\delta}(\Lambda^2\dot\Sigma)
  \end{eqnarray*}
  for any $l$, which is trivial, as $\tilde D$ satisfies a maximum
  principle and the weight $-\tilde\delta>0$. Thus $\tilde D$ is
  surjective.  In particular, any non--trivial solution of $\tilde D f=0$ is
  asymptotic to the eigenvector $\hat e$ of eigenvalue $\hat\mu$ at
  each puncture.

  The same argument shows that $\tilde D_r f=\ast d(df\circ
  j+2f\,\kappa_r)$ is surjective for any $r$, where $\kappa_r$ is a
  family of 1--forms so that the asymptotic operator of the associated
  first order equation has fixed eigenvalues 0 and
  $\hat\mu=\hat\lambda-\lambda\le 0$ for the eigenvectors with winding
  zero. By standard theory (see \cite{hofer_embedding}) no other
  eigenvalues of the operator in this family can enter the interval
  $[\hat\mu,0]$. 

  Thus solutions to $\tilde D_rf=0$ are non--zero in some neighborhood
  of the punctures, where the neighborhood may depend on $r$. We now
  show that non--trivial solutions $f$ have no zeros.

  Let $\kappa_r$, $r\in[0,1]$ be a family of 1--forms with
  $\kappa_1=\kappa$ and $\kappa_0=\beta(s)\frac{\hat\mu}{2}\,dt$, where
  $\beta:\dot\Sigma\to[0,1]$ is a cutoff function supported on the
  neck regions adjacent to the punctures and equal to 1 in some
  neighborhood of the punctures and monotone on each neck. Then
  \begin{eqnarray*}
    \tilde D_0f=-\Delta f+\hat\mu\del_s(\beta\,f)
  \end{eqnarray*}
  has kernel given by functions that are equal to some constant $c$ on the thick part
  and equal to
  \begin{eqnarray*}
    f(s,t)=c\,e^{\hat\mu\int_0^s \beta(s')ds'}
  \end{eqnarray*}
  on the necks, remembering that $f\in W^{k,p}_{\tilde\delta}$. In
  particular, the kernel is 1--dimensional so $\tilde D$ has index 1,
  and non--trivial solutions to $\tilde D_0f=0$ have no zeros. Now fix
  a point $p$ in the tick part and consider the family of solutions
  $f_r$ of $\tilde D_rf_r=0$ with $f_r(p)=1$. All $f_r$ are positive
  in some neighborhood (depending on $r$) of the punctures. We claim
  that $f_r\ne 0$ for any $r\in[0,1]$. If not, there exists a smallest
  $r>0$ and so that $f_r>=0$ and $f_r(z)=0$ for some $z\in
  \dot\Sigma$. But this is impossible by Harnack's inequality. Thus
  $f_r>0$ on $\dot\Sigma$ for all $r\in[0,1]$.
\end{proof}

Lemma \ref{lem:positive_solution} gives the local model for families
of nicely immersed curves. The following Theorem makes this precise.
\begin{theorem}\label{thm:local_foliation}
  Let $v:\dot\Sigma\to Z$ be a nicely immersed $\H$--holomorphic
  map. Then there exists a smooth 1--dimensional family of
  $\H$--holomorphic maps which is, up to gauge transformations, given
  by translation along the characteristic direction by a non--zero
  function. If $v$ is nicely embedded and asymptotic to a collection
  of closed characteristics $B$ then the family locally foliates
  $Z\setminus B$.
\end{theorem}

\begin{proof}
  Let $(\xi,h)$ be in the kernel of the linearized operator. By Lemma
  \ref{lem:lin_op2} we can choose an infinitesimal gauge
  transformation $\nu$ so that $\xi=dv(\nu)+f\,R$, where $f$ satisfies
  $\tilde D\,f=0$ and $\tilde D$ is the operator from Equation
  (\ref{eq:tildeD}).

  By Lemma \ref{lem:positive_solution} the operator $\tilde D$ is
  surjective and elements in the kernel extend continuously to the
  radial compactification at the punctures, they are unique up to
  scaling, and non--trivial solutions have no zeros.

  By Lemma \ref{lem:lin_op2} there is a bijective correspondence
  between elements in the kernel of the system linearized equations
  from Lemma \ref{lem:linearization} and elements in the kernel of
  $\tilde D$. The surjectivity of $\tilde D$ then implies that the
  system of linearized equations is surjective and $v$ lives in a
  1--parameter family. By Lemma \ref{lem:positive_solution} we see
  that the 1--dimensional space of solution is infinitesimally
  obtained from one another by ``translating'' in the characteristic
  direction by the nowhere zero bounded function $f$, up to gauge
  transformation.

  Let $v_t$, $t\in I$ denote such a local family of solutions with
  $v_0=v$.  If $v$ is embedded, then so are nearby curves in this
  family, and since the function $f$ has no zeros there exists
  $\varepsilon>0$ so that $v_t$ are mutually disjoint for
  $t\in(-\varepsilon,\varepsilon)$. Thus the map $\hat
  v:(-\varepsilon,\varepsilon)\times \dot\Sigma\to Z\setminus B$ is a
  diffeomorphism onto its image, so solutions foliate a tubular
  neighborhood of the image curve $v$ in $Z\setminus B$. 
\end{proof}

\section{Global Theory and Compactness}
The local foliation results from the previous section lead to
no-first-intersection results that in turn allow to show that the
periods of families of nicely embedded curves are uniformly bounded
and thus yields compactification of connected components of the moduli
space of nicely embedded curves \ref{thm:moduli_space}. 

When analyzing nicely embedded maps we will frequently choose a
complement $\tilde \H$ of the coexact forms in the coclosed forms that
is supported away from the punctures and use this to lift maps to the
symplectization. The definition of $\tilde \H$ depends on the choice
of domain complex structure $j$, so when considering families of maps
and complex structures $j_s$ we may consider families of complements
$\tilde \H_s$ that are supported away from the punctures and sometimes
also away from neighborhoods of other points in the interior of
$\dot\Sigma$.  The significance of choosing the support away from the
punctures or other points is so that the lifts are then locally
$J$--holomorphic and the standard theory applies there.

Similar to the $J$--holomorphic case, we can prove a ``no first
intersections''--type result for families of nicely embedded
$\H$--holomorphic maps.
\begin{lemma}\label{lem:first_intersection}
  For $t\in I=(a,b)\subset\setR$, let $v_t:(\dot\Sigma,j_t)\to Z$ be a
  family of somewhere injective index 1 $\H$--holomorphic maps with
  smooth domain complex structures $j_t$ and assume $v_t$ is nicely
  embedded for some $t\in I$. Then all $v_t$ are nicely embedded for
  $t\in I$, and if $v_{t_0}(z_0)=v_{t_1}(z_1)$ for some $t_0,t_1\in I$
  and $z_0,z_1\in\dot\Sigma$, then the images of $v_0$ and $v_1$
  coincide.
\end{lemma}
\begin{proof}
  First we show that all maps $v_t$ are embedded. Suppose $I_0\subset
  I$ is a maximal interval so that all maps in $I_0$ are embedded. The
  interval $I_0$ is open and by assumption not empty. If $I=I_0$, then
  all maps are embedded and there is nothing to prove. Otherwise let
  $s\in \del I_0\cap I$.  By possibly perturbing the almost complex
  structures in the neighborhood of the image of an injective point of
  $v_s$ and considering the corresponding family of maps that are
  $\H$--holomorphic with respect to that almost complex structure, we
  may assume that the self--intersections of $v=v_s$ are occurring in
  the union of some disjoint open balls in $\dot\Sigma$.

  Let $\tilde U\subset\dot\Sigma$ be the union of disjoint open balls
  containing the self--intersection locus and the punctures and denote
  the union of slightly bigger disjoint open balls containing $\tilde
  U$ by $U$. Choose a complement $\tilde \H$ of the coexact forms in
  the coclosed forms so that its elements are supported off $U$. Let
  $(a,v)$ be a lift of $v$ to the symplectization with respect to
  $\tilde \H$, i.e. $v^\ast\alpha+da\circ j\in\tilde \H$. By possibly
  modifying $\tilde \H$ by adding a function that is supported in the
  $U$ we may assume that $\tilde \H$ is supported off $\tilde U$ and
  the lift $(a,v)=(a_s,v_s)$ has at least one self--intersection.

  Now extend $\tilde \H=\tilde \H_s$ to a family $\tilde H_t$ of
  complements of the coexact forms in the coclosed forms
  (w.r.t. $j_t$) that is supported off of $\tilde U$. We use
  $\tilde\H_t$ to lift the entire family $v_t$ to the symplectization
  by choosing a family of functions $a_t:\dot\Sigma\to Z$ so that
  $v_t^\ast\alpha-da_t\circ j\in \tilde \H_t$ that extend the lift
  $(a_s,v_s)$. The maps $\tilde v_t=(a_t,v_t)$ are $J$--holomorphic on
  $\tilde U$, so the self--intersection number of $\tilde v_s$ is
  positive. But $\tilde v_t$ has self--intersection number 0 for $t\in
  I_0$. Moreover, all maps $\tilde v_t$ are $J$--holomorphic in a
  neighborhood of the punctures and the Sobolev weights have been
  chosen to be maximal. But this contradicts that in this case the
  intersection number is a topological invariant by
  \cite{siefring}. So $v_t$ is embedded for all $t\in I$.

  Next suppose that there are $t_0,t_1\in I$ and
  $z_0,z_1\in\dot\Sigma$ so that $v_{t_0}(z_0)=v_{t_1}(z_1)$ but the
  images of $v_{t_0}$ and $v_{t_1}$ don't coincide. Let $s\in
  (t_0,t_1]$ so that $v_{t_0}$ is disjoint from $v_t$ for all $t\in
  (t_0,s)$ but $v_{t_0}$ and $v_s$ intersect. If $v_{t_0}$ and $v_s$
  have the same image we are done. If not we may again assume, by
  possibly perturbing the almost complex structure in a neighborhood
  of a point on the image of $v_s$ that is disjoint from $v_{t_0}$,
  that all intersections of $v_{t_0}$ and $v_s$ occur in a union of
  disjoint balls. Just as above we may choose complements of the
  coexact forms in the coclosed forms $\tilde \H_{t_0}=\tilde
  \H_{t_0}(j_{t_0})$ and $\tilde \H_s=\tilde \H_s(j_s)$ so that the
  corresponding lifts to the symplectization $\tilde v_{t_0}$ and
  $\tilde v_s$ are $J$--holomorphic in a neighborhood of the punctures
  and the points where $v_0$ and $v_s$ intersect.  Extend these
  complements to a smooth family $\tilde \H_t=\tilde \H_t(j_t)$ for
  all $t\in I$. It follows from \cite{siefring} that the intersection
  number of maps in this family is a topological invariant since the
  weights have been chosen to be maximal. For $t\in (t_0,s)$ the
  intersection number of $\tilde v_t$ with $\tilde v_{t_0}$ is zero,
  contradiction that $\tilde v_{t_0}$ intersects $\tilde v_s$
  positively, after possibly adding a constant function to the
  $\setR$--factor of the map.

  In conjunction we conclude that all maps $v_t$ are embedded and
  disjoint. 
\end{proof}
The Lemma explicitly excludes the cases that the complex structure
becomes degenerate, or that maps in the family are multiply
covered. The Lemma follows more easily from recent work of
R. Siefring (Theorem 2.10 of \cite{siefring_intersection}).

We need the following standard definition. 
\begin{definition}
  An $\H$--holomorphic map $v:\dot\Sigma\to Z$ asymptotic to a
  collection of non--degenerate closed characteristics $B$ is called
  {\em a surface of section} if $v$ is embedded and every
  characteristic flow line in $Z$ is either an asymptotic orbit for
  $v$ or intersects the image of $v$ in forward and backward time.
\end{definition}

\begin{lemma}\label{lem:embedded_periods}
  Let for $t\in I=(a,b)\subset\setR$, let $v_t:\dot\Sigma\to Z$ be a
  smooth family of nicely embedded $\H$--holomorphic maps so that for
  some $c\in I$ $v_c$ is a surface of section. Then all $v_t$, $t\in
  I$ are surfaces of section and the periods of the family $v_t$ are
  uniformly bounded.
\end{lemma}
\begin{proof}
  By Lemma \ref{lem:positive_solution} the family $f_t$ is obtained
  from $v_c$ by shifting in the characteristic direction, so all $v_t$
  are surfaces of section.
  
  We are left to show that the periods are uniformly bounded. Since $v_c$ is a
  surface of section we can define the ``characteristic width'' $\hat T$ of
  $Z\setminus v_c$ as
  \begin{eqnarray*}
    \hat T=\hat T(v_c)=\sup_{z\in Z}\inf_{t\in (0,\infty)}\left\{t\,\Big|\,
      \phi_{t}(z)\in \mathrm{image}(v_c) \right\} .
  \end{eqnarray*}
  
  Now suppose that the periods of the $v_n=v_{t_n}$ are unbounded for
  some sequence $t_n\in I$. Choose a subsequence of $v_n$ so that the
  corresponding domain complex structures converge in some
  $\overline{\M}_{g,n}$. By the usual bubbling off analysis we may
  assume without loss of generality that there exists a constant $C>0$
  so that
  \begin{eqnarray*}
    |\pi_F\,dv_n(z)|\le C(1+|v_n^\ast\alpha(z)|)\le C^2(1+|\eta_n(z)|),\qquad
    \forall\,z\in\dot\Sigma.
  \end{eqnarray*}
  If the first inequality were not true there would be a sequence of
  points $p_n$ and a subsequence with
  $|\pi_F\,dv_n(p_n)|>n(1+|v_n^\ast\alpha(p_n)|)$, which, after
  rescaling, yields a bubble with non--trivial $\omega$--energy, which
  can only happen finitely many times.  If the first inequality holds
  true, but the second inequality does not, so $|\eta_n(p_n)|$ grows
  slower than the coexact part $|da_n(p_n)|$ and by the same argument
  we obtain a bubble map with non--trivial $\omega$--energy, which
  again can only happen finitely many times.

  We split the remaining argument into two cases. Either $dv_n$, and
  thus $\eta_n$, becomes unbounded in the cylindrical metric or
  $v_{n}$ remains bounded but the twist of some neck becomes
  unbounded. In the former case we have that $\eta_n(z)$ becomes
  unbounded in some open set $U\subset\dot\Sigma$ by Harnack's
  inequality. Then we may assume that there exists a sequence of
  points $p_n\in U$ so that $v_{n}^\ast\alpha(p_n)$ grows faster than
  $\pi_F\,dv_{t_n}(p_n)$, otherwise the $\omega$--energy would become
  unbounded. Extracting a convergent subsequence of maps centered at
  $p_n$ and rescaled so that $v_{t_n}^\ast\alpha$ has norm 1 we see
  that this subsequence converges uniformly on compact subsets of
  $\setC$ to a characteristic flow line of unit speed. In particular,
  this subsequence must intersect $v_c$ since $\hat T$ is finite.  But
  this is impossible by Lemma \ref{lem:first_intersection}. Similarly,
  if $dv_{n}$ remains uniformly bounded, but the twist becomes
  unbounded, we may extract a subsequence so that the images of
  subsets of the necks $[-R_n,R_n]\times \{1\}$ converge uniformly to
  a characteristic flow line of unbounded length, again forcing an
  intersection with $v_c$, contradicting Lemma
  \ref{lem:first_intersection}.
\end{proof}

\begin{theorem}\label{thm:moduli_space} 
  Suppose all closed characteristics of $(Z,\alpha,\omega)$ are
  non--degenerate and let $v:\dot\Sigma\to Z$ be an $\H$--holomorphic
  nicely embedded surface of section asymptotic to a collection of
  closed characteristics $B$.

  Let ${\M}_v^1={\M}_v^1(\alpha,\omega,J)$ be the connected component
  of the moduli space of embedded $\H$--holomorphic maps with one
  marked point containing $v$.

  Then ${\M}_v^1$ has a natural compactification
  $\overline{{\M}_v^1}$. If $\overline{{\M}_v^1}$ does not have
  boundary, then the evaluation map gives a diffeomorphism
  $ev:{\M}_v^1\to Z\setminus B$.
\end{theorem}

\begin{proof}
  ${\M}_v^1$ has a natural compactification $\overline{{\M}_v^1}$ by
  Lemma \ref{lem:embedded_periods} and Theorem \ref{thm:compactness}.
  
  Let $f:\overline{\M_v^1}\to\overline{\M_v}$ be the forgetful
  map. Then $\overline{\M_v}$ is a closed, connected 1--dimensional
  manifold.

  Assume it does not have boundary, then $\M_v^1=\overline{\M_v^1}$.
  Let $u\in \M_v$. By Theorem \ref{thm:local_foliation} $u$ is part of
  a unique local 1--parameter family $u_t$ of maps in $\M_v$ with
  $u=u_0$ that are all embedded and that locally foliates $Z\setminus
  B$. By Lemma \ref{lem:first_intersection} no two maps in $\M_v$ can
  intersect each other, so the image of the evaluation map is
  injective. Since $v$ is a surface of section, every characteristic
  flow line intersects $v$ in finite time. Then the fact that $\M_v$
  is compact, together with the strong version of the implicit
  function theorem Theorem \ref{thm:local_foliation}, implies that the
  image of the evaluation map from $\M_v^1$ to $Z\setminus B$ is
  surjective.
\end{proof}
This theorem suggests the existence of a finite energy foliation in
the presence of a nicely embedded surface of section $v$. The proof of
this requires a gluing theorem to continue the moduli space
$\overline{\M_v}$ past the boundary. This requires a slightly
different pre--gluing spaces compared to the $J$--holomorphic case,
due to the different behavior of neck--maps and is work in progress.

\section{$\H$--Holomorphic Open Book Decompositions}
\label{sec:h-holomorphic-open}
As an application we now prove the existence of $\H$--holomorphic open
book decomposition supported by any given contact structure. Armed
with the results of Section \ref{sec:local-fredh-theory} this is a
straightforward extension of the work of Wendl in
\cite{wendl_open_book}.

\begin{theorem}\label{thm:open_book}
  Suppose $(Z,F)$ is a closed 3--manifold with positive, co-oriented
  contact structure $F$ and $\pi:Z\setminus B\to S^1$ is an open book
  decomposition that supports $F$. Then, after an isotopy of $\pi$,
  there exists a nondegenerate contact form $\alpha$ with
  $\ker\alpha=F$, a compatible almost complex structure $J$ on $F$ and
  a smooth $S^1$--family of nicely embedded $\H$--holomorphic maps
  parametrizing the pages of $\pi$.
\end{theorem}

\begin{proof}
  The proof follows the proof of the planar case given in
  \cite{wendl_open_book}. As explained in the first part of Section 3
  of \cite{wendl_open_book} there exists a stable Hamiltonian
  structure $(F,\alpha_0,\omega_0,J_0)$, and a small isotopy of $\pi$
  so that the pages of the open book decomposition are
  $J_0$--holomorphic. We now wish to deform this stable Hamiltonian
  structure to one arising from a contact form $\alpha$ with contact
  structure $F$ and some compatible almost complex structure $J$, and
  we wish to deform the foliation to an $\H$--holomorphic foliation of
  $(Z\setminus B)$ with respect to the almost complex structure $J$ on
  $F$. Any $J_0$--holomorphic map is automatically $\H$--holomorphic
  (w.r.t. $(\alpha_0,\omega_0,J_0)$) and that all pages are nicely
  embedded. Denote this connected component of the moduli space of
  $\H$--holomorphic maps by $\M_0$ and the space of maps with one
  marked point on the domain by $\M_0^1$.

  Using the implicit function theorem on the compact family of curves,
  for a sufficiently small perturbation of stable Hamiltonian
  structure that leaves a neighborhood of the binding invariant, we
  can find nearby $\H$--holomorphic maps. More precisely, we obtain
  families $\{\M_t\}_{t\in[0,\varepsilon)}$ of moduli spaces of
  $\H$--holomorphic maps, where $\M_t$ consists of $\H$--holomorphic
  maps with respect the the stable Hamiltonian structure
  $(\alpha_t,\omega_t,J_t)$, $t\in[0,\varepsilon]$ from
  \cite{wendl_open_book} in the connected component containing $\M_0$.
  
  By possibly shrinking $\varepsilon$ we may assume that all elements
  of $\M_t$ are nicely embedded, and thus locally foliating by
  Theorem \ref{thm:local_foliation}. Thus $\M_t$ foliates
  $Z\setminus B$ for $\lambda\in[0,\varepsilon]$ by Theorem
  \ref{thm:moduli_space}. 
   
  By Theorem \ref{thm:moduli_space} we see that each $\M_t^1$
  is diffeomorphic to $Z\setminus B$ via the evaluation map and that
  $\M_t$ is diffeomorphic to $S^1$. The characteristic vector
  field $R_t$ is transverse to each map in $\M^1_t$ since the
  maps are embedded. Thus the forgetful map
  $\pi_t^1:\M_t^1\to \M_t\approx S^1$ give the desired
  isotopy of open book projections.

  By Giroux's classification, the contact structures $F_t=\ker\alpha_t$
  are contactomorphic for any $0<t\le\varepsilon$ since they are
  supported by the same open book decomposition. In particular, all
  are contactomorphic to $F$.
\end{proof}

\providecommand{\bysame}{\leavevmode\hbox to3em{\hrulefill}\thinspace}
\providecommand{\MR}{\relax\ifhmode\unskip\space\fi MR }
\providecommand{\MRhref}[2]{%
  \href{http://www.ams.org/mathscinet-getitem?mr=#1}{#2}
}
\providecommand{\href}[2]{#2}

\end{document}